\documentclass[10pt]{amsart}
\usepackage{amsmath,amsfonts,amsthm,amssymb,amscd, soul, color}
\def\classification#1{\def\@class{#1}}
\classification{\null}

\DeclareFontFamily{OT1}{rsfs}{}
\DeclareFontShape{OT1}{rsfs}{n}{it}{<-> rsfs10}{}
\DeclareMathAlphabet{\mathscr}{OT1}{rsfs}{n}{it}

\newcommand{\R}{{\mathbb R}}
\newcommand{\Pro}{{\mathbb P}}

\newcommand{\C}{\mathbb{C}}
\newcommand{\F}{\mathbb{F}}

\newcommand{\Ge}{\mathcal{G}}
\newcommand{\Ka}{\mathcal{K}}

\newcommand{\ee}{\boldsymbol e}
\newcommand{\om}{\boldsymbol \omega}
\newcommand{\qu}{\boldsymbol q}
\newcommand{\ve}{\boldsymbol v}
\newcommand{\ur}{\boldsymbol u}

\newtheorem{theorem}{Theorem}

\newtheorem{lemma}[theorem]{Lemma}
\newtheorem{corollary}[theorem]{Corollary}
\theoremstyle{remark}
\newtheorem{remark}[theorem]{Remark}

\title{On the number of incidences between points and planes in three dimensions}

\author{Misha Rudnev}
\address{Misha Rudnev, Department of Mathematics, University of Bristol,
  Bristol BS8 1TW, United Kingdom}
\email{m.rudnev@bristol.ac.uk}

\subjclass[2000]{68R05,11B75}
\begin{document}
\begin{abstract} We prove an incidence theorem for points and planes in the projective space $\Pro^3$ over any field $\mathbb F$, whose characteristic $p\neq 2.$ An incidence is viewed as an intersection along a line of a pair of  two-planes from two canonical rulings of the Klein quadric. The Klein quadric can be traversed by a generic hyperplane, yielding a line-line incidence problem in a three-quadric,  the Klein image of a regular line complex. This hyperplane can be chosen so that  at most two lines meet. Hence, one can apply an algebraic  theorem of Guth and Katz, with a constraint involving $p$ if $p>0$.

This yields a bound on the number of incidences between $m$ points and $n$ planes in $\Pro^3$, with $m\geq n$ as
$$O\left(m\sqrt{n}+ m k\right),$$
where $k$ is the maximum number of collinear  planes, provided  that $n=O(p^2)$ if $p>0$. Examples  show that  this bound cannot be improved without additional assumptions.

This gives  one a vehicle to establish geometric incidence estimates when $p>0$. For a non-collinear point set $S\subseteq \F^2$ and a  non-degenerate symmetric or skew-symmetric bilinear form $\omega$, the number of distinct values of $\omega$ on pairs of points of $S$ is $\Omega\left[\min\left(|S|^{\frac{2}{3}},p\right)\right]$. This is also the best known bound over $\R$, where it follows from the Szemer\'edi-Trotter theorem. Also, a set $S\subseteq \mathbb F^3$, not supported  in a single semi-isotropic plane contains a point, from which $\Omega\left[\min\left(|S|^{\frac{1}{2}},p\right)\right]$ distinct distances to other points of $S$ are attained.

\end{abstract}

\maketitle

\section{Introduction} 
Let $\F$ be a field of characteristic $p$ and $\Pro^d$ the $d$-dimensional projective space over $\F$. Our methods do not work for $p=2$, but the results in view of constraints in terms of $p$ hold trivially for $p=O(1)$. As usual, we use the notation $|\cdot|$ for cardinalities of finite sets. The symbols $\ll$, $\gg,$ suppress absolute constants in inequalities, as well as respectively do $O$ and $\Omega$. Besides,  $X=\Theta(Y)$ means that $X=O(Y)$ and $X=\Omega(Y)$. The symbols $C$ and $c$ stand for absolute constants, which may sometimes change from line to line, depending on the context. When we turn to sum-products, we use the standard notation
$$A+B=\{a+b:\,a\in A,\,b\in B\}$$ for the sumset $A+B$ of $A, B\subseteq \mathbb F$, similarly for the product set $AB$. 

\medskip

The Szemer\'edi-Trotter theorem \cite{ST} on the number of incidences between lines and points in the Euclidean plane has many applications in  combinatorics. The theorem is also valid over $\C$, this was first proved by T\'oth \cite{T}. In positive characteristic, however, no universal satisfactory for applications point-line incidence  estimate is available. The current ``world record" for partial results in this direction for the prime residue field $\F_p$ is due to Jones \cite{J}.

This paper shows that in three dimensions there is an incidence estimate
between a set $P$ of  $m$ points and a set $\Pi$ of $n$ planes in $\Pro^3$, valid for any field of characteristic $p\neq 2$.  If $p>0$, there is a constraint that  $\min(m,n)=O(p^2).$ Hence, the result is trivial, unless $p$ is regarded as a large parameter.  Still, since our geometric set-up in terms of $\alpha$-and $\beta$-planes in the Klein quadric breaks down for $p= 2$, we have chosen to state that $p\neq 2$ explicitly in the formulation of main results.  Extending the results to more specific situations, when the constraint in terms of $p$ can be weakened may not be impossible but way beyond the methodology herein. A few more words  address this issue in the sequel.

\medskip
The set of incidences is defined as
\begin{equation}
I(P,\Pi) :=\{(q,\pi)\in P\times \Pi:\,q\in\pi\}.
\label{ins}\end{equation}

Over the reals, the point-plane incidence problem has been studied quite thoroughly throughout the past 25 years and several tight bounds are known. In general one can have all the points and planes involved to be incident to a single line in space, in which case the number of incidences is trivially $mn$. To do better than that, one needs  some non-degeneracy assumption regarding collinearity, and the results quoted next differ as to the exact formulation of such an assumption.

In the 1990 paper of Edelsbrunner et al. \cite{EGS} it was proven (modulo slow-growing factors that can be removed, see \cite{AS})  that if no three planes are collinear in $\mathbb R^3$,
\begin{equation}\label{egsest}
|I(P,\Pi)| = O\left(m^{\frac{4}{5}}n^{\frac{3}{5}}+m+n\right).
\end{equation}
This bound was shown to be tight for a wide range of $m$ and $n$, owing to a construction by Brass and Knauer \cite{BK}. A thorough review of the state of the art by the year 2007 can be found in the paper of Apfelbaum and Sharir \cite{AS}.

Elekes and T\'oth \cite{ET} weakened the non-collinearity assumption down to that all planes were  ``not-too-degenerate". That is a single line in a plane may support only a constant proportion of incidences in that plane. They proved a bound
\begin{equation}\label{etest}
|I(P,\Pi)| = O\left((mn)^{\frac{3}{4}} + m\sqrt{n} + n\right)
\end{equation}
and presented a construction, showing it to be generally tight. The constructions supporting the tightness of both latter estimates are algebraic and extend beyond the real case.

More recently, research in incidence geometry over $\R$ has intensified after the introduction of the polynomial partitioning technique in a breakthrough paper of Guth and Katz \cite{GK}. E.g., there is now  a ``continuous" generalisation of the bound \eqref{egsest} by Basit and Sheffer \cite{BS}:

\begin{equation}\label{bsest}
|I(P,\Pi)| = O^*\left(m^{\frac{4}{5}+\epsilon}n^{\frac{3}{5}}k^{\frac{2}{5}} +mk+n\right),
\end{equation}
where $k$ is the maximum number of collinear planes. For any $\epsilon>0$, the constant hidden in the $O^*$-symbol depends on $\epsilon$.

The proofs of the above results rely crucially on the order properties of $\mathbb R$. Some of them, say  \eqref{etest} extend over $\mathbb C$, for it is based on the Szemer\'edi-Trotter theorem. Technically harder partitioning-based works like \cite{BS} have so far defied generalisation beyond $\R$.

\medskip
This paper presents a different approach to point-plane incidences in the projective three-space $\Pro^3$. The approach appears to be robust enough to embrace, in principle, all fields $\F$, but for the apparently special case of characteristic $2$. When we have a specific field  $\F$ in mind, we use the notation $\mathbb{FP}$ for the projective line $\Pro$. The novelty of our approach is on its geometric side: we fetch and use extensively the classical XIX century Pl\"ucker-Klein formalism for line geometry in $\Pro^3$. This is combined with a recent algebraic incidence theorem for counting line-line intersections in three dimensions by Guth and Katz.

The work of Guth and Katz, see  \cite{GK} and the references contained therein for its predecessors, established  two important theorems. Both rested on the polynomial Nullstellensatz principle, which was once again demonstrated to be so efficient a tool for discrete geometry problems by Dvir, who used it to resolve the finite field Kakeya conjecture \cite{D}. The proof of the first Guth-Katz theorem, Theorem 2.10 in \cite{GK}, was in essence algebraic, using the Nullstellensatz  and basic properties of ruled surfaces, which come into play due to the use of the classical XIX century geometry Monge-Salmon theorem. See  \cite{Sa} for the  original exposition of the latter theorem, as well as  \cite{K} (wholly dedicated to the prominent role this theorem plays in today's incidence geometry) and Appendix in \cite{Ko}.

The second Guth-Katz theorem, Theorem 2.11 in \cite{GK}, introduced the aforementioned  method of polynomial partitioning of the real space, based on the Borsuk-Ulam theorem. It is the latter theorem of Guth and Katz that has recently attracted more attention and follow-ups. Since we work over any field, we cannot not use polynomial partitioning.

It is a variant of Theorem 2.10 from  \cite{GK} that plays a key role here, and is henceforth referred to as {\em the} Guth-Katz theorem. We share this, at least in part, with a recent work of Koll\'ar \cite{Ko} dedicated to point-line incidences in $3D$, in particular over fields with positive characteristic. 

\begin{theorem}[Guth-Katz] \label{gkt} Let $\mathcal L$ be a set of $n$ straight lines in $\R^3$. Suppose,  no more then two lines are concurrent. Then the number of pair-wise intersections of lines in $\mathcal L$ is 
$$
O\left(n^{\frac{3}{2}}+ nk\right),
$$
where $k$ is the maximum number of lines, contained in a plane or regulus.
\end{theorem}

\medskip
The proof of Theorem \ref{gkt} goes about over the complex field.\footnote{For a reader not familiar with the proof of Theorem \ref{gkt}, that is Theorem 2.10 in \cite{GK}, we recommend  Katz's note \cite{K} for more than an outline of the proof. 
See also a post  {\sf www.terrytao.wordpress.com/2014/03/28/the-cayley-salmon-theorem-via-classical-differential-geometry/} by Tao and the links contained therein.} Moreover, it extends without major changes to any algebraically closed field, under the constraint $n=O(p^2)$ in  the positive characteristic case.
This was spelt out by Koll\'ar,  see \cite{Ko} Corollary 40,  with near-optimal values of constants.

To complete the introduction, let us briefly discuss, in slightly more modern terms, the ``continuous'' Monge-Salmon theorem, brought in by Guth and Katz to  discrete geometry. Suppose the field $\mathbb F$ is  algebraically closed field and $Z$ is a  surface in $\F\Pro^3$, defined as the zero set of a minimal polynomial $Q$ of degree $d$. A point $x\in Z$ is called {\em flechnodal} if there is a line $l$ with at least fourth order contact with $Z$ at $x$, that is apart from $x\in l$, at least three derivatives of $Q$ in the direction of $l$ vanish at $x$. Monge  showed that
flechnodal points are cut out by a homogeneous polynomial, whose degree Salmon claimed to be  equal to $11d-24$ (which is sharp for $d=3$, due to the celebrated Cayley-Salmon theorem). Thus, for an irreducible $Z$, either all points are flechnodal, or flechnodal points lie on a curve of degree $d(11d-24)$.  Over the complex field Salmon proved that assuming that all points of $Z$ are flechnodal implies that $Z$ is ruled. In positive characteristic it happens that there exist high degree non-ruled surfaces, where each point is flechnodal. But not for $d<p$. Voloch \cite{V} adapted the Monge-Salmon proof to modern terminology, $p>0$, $d<p$, and also conjectured that counterexamples may take place  only if $p$ divides $d(d-1)(d-2)$. 

The following  statement is implicit in the proof of Proposition 1 in \cite{V}.

\begin{theorem}[Salmon] \label{Salmon} An irreducible algebraic surface in $\Pro^3$ over an algebraically closed field $\F$, containing more than $d(11d-24)$ lines must be ruled, under the additional constraint that $d<p$ if $\F$ has positive characteristic $p$.
\end{theorem}

Once Theorem \ref{Salmon} gets invoked within the proof of Theorem \ref{gkt}, the rest of it uses  basic properties of ruled surfaces,  for a summary see \cite{Ko}, Section 7. We complement it with some additional background material in Section \ref{ruled}, working with the Grassmannian parameterising the set of lines in $\Pro^3$, that is the Klein quadric.

\section{Main results}
The main geometric idea of this paper is to interpret incidence problems between points and planes in $\Pro^3$ as line-line incidence problems in a projective three-quadric $\Ge$. $\Ge$ is contained in the Klein quadric $\Ka$ representing the space of lines in the ``physical space'' $\Pro^3$ in the ``phase space'' $\Pro^5$. $\Ge$ is the Klein image of a so-called {\em regular line complex} and has many well-known geometric properties.  In comparison to $\Pro^3$, where the space of lines is four-dimensional, the space of lines in $\Ge$ is three-dimensional, and this enables one to satisfy  the no-multiple-concurrencies hypotheses of Theorem \ref{gkt}.  It will also turn out that the parameters denoted as $k$  in both the point-plane incidence estimate (\ref{bsest}) and Theorem \ref{gkt} are closely related.

Our main result is as follows.

\begin{theorem}  \label{mish} Let $P, \Pi$ be sets of points and planes, of cardinalities respectively $m$ and $n$ in  $\Pro^3$. Suppose, $m\geq n$ and if $\F$ has positive characteristic $p$, then $p\neq 2$ and $n=O(p^2)$. Let $k$ be the maximum number of collinear planes.

Then
\begin{equation}\label{pups}
|I(P,\Pi)|=O\left( m\sqrt{n}+ km\right).\end{equation}
\end{theorem}
The statement of the theorem can be reversed in an obvious way, using duality in the case when the number of planes is greater than the number of points. Note that the $km$ term may dominate only if $k\geq\sqrt{n}$.

The estimate (\ref{pups}) of Theorem \ref{mish} is a basic universal estimate. It is weaker than the above quoted estimates (\ref{egsest}), as well as  (\ref{bsest}) for small values of $k$, and slightly weaker than
(\ref{etest}). Later in Section \ref{example}, for completeness sake, we present a construction, not so dissimilar from those in \cite{BK} and \cite{ET}, showing that in the case $n=m$ and $k=m^{\frac{1}{2}}$, the estimate (\ref{pups}) is tight, for any admissible $n$. 

\begin{remark}\label{sharps} Let us argue that in positive characteristic and  solely under the constraint $\min(m,n)=O(p^2)$ the main term in the estimate  \eqref{pups} cannot be improved. This suggests that analogues of stronger Euclidean point-plane incidence bounds like \eqref{egsest} do not extend to positive characteristic without additional assumptions stronger than in Theorem \ref{mish}.

Let $\F=\F_p$,  and take the point set $P$ as a smooth cubic surface in $\F^3$, so $|P|=O(p^2)$. Suppose, $|\Pi|>|P|$, so the roles of $m, n$ in Theorem \ref{mish} get reversed. A generic plane intersects $P$ at $\Omega(p)$ points.  By the classical Cayley-Salmon theorem (which follows from the statement of Theorem \ref{Salmon} above)   $P$ may contain at most $27$ lines. Delete them, still calling $P$ the remaining positive proportion of $P$. Now no more than three points in $P$ are collinear, so $k=3$. However, for a set of  generic planes $\Pi$, $|I(P,\Pi)|=\Omega(|\Pi|\sqrt{|P|})$, which matches up to constants the bound \eqref{pups} when $|\Pi|>|P|$.
\end{remark}

We also give  two applications of Theorem \ref{mish} and
show how it yields reasonably strong geometric incidence estimates over fields with positive characteristic. The forthcoming Theorem \ref{spr} claims that any plane set $S\subset \F^2$ of $N$ non-collinear points determines $\Omega\left[\min\left(N^{\frac{2}{3}},p\right)\right]$ distinct pair-wise bilinear -- i.e., wedge or dot -- products, with respect to any origin.

If $S=A\times A$, $A\subseteq \F$ this improves to a sum-product type inequality
\begin{equation}
|AA+AA|=\Omega\left[\min\left(|A|^{\frac{3}{2}},p\right)\right].
\label{2spr}\end{equation}
In the special case of $A$ being a multiplicative subgroup of $\F^*_p$, the same bound was proved by Heath-Brown and Konyagin \cite{HBK} and improved by V'jugin and Shkredov \cite{VS} (for suitably small multiplicative subgroups) to $\Omega\left(\frac{|A|^{\frac{5}{3}}} {\log^{\frac{1}{2}}|A|}\right).$
Theorem \ref{mish} becomes a vehicle to extend bounds for multiplicative subgroups  to approximate subgroups. 

For more applications of Theorem \ref{mish}  to questions of sum-product type see \cite{RRS}. Results in the latter paper include a new state of the art sum-product estimate
$$\max(|A+A|,\,|AA|)\gg |A|^{\frac{6}{5}},\qquad\mbox{for }|A|<p^{\frac{5}{8}},$$ obtained from Theorem \ref{mish} in a manner, similar to how the sum-product  exponent $\frac{5}{4}$ gets proven over $\R$ using the Szemer\'edi-Trotter theorem in the well-known construction by Elekes \cite{E0}. The previously known best sum-product exponent $\frac{12}{11}-o(1)$ over $\F_p$ was proven by the author \cite{R}, ending a stretch of many authors' incremental contributions based on the so-called {\em additive pivot} technique introduced by Bourgain, Katz and Tao \cite{BKT}.\footnote{In a forthcoming paper with E. Aksoy, B. Murphy, and  I. D. Shkredov we present further applications of Theorem \ref{mish} to sum-product type questions in positive characteristic.}

Such reasonably strong bounds in positive characteristic have been available so far only for subsets of finite fields, large enough relative to the size of the field  itself: see, e.g., \cite{HI}. Theorem \ref{mish} enables one to extend these bounds to small sets, and the barrier it imposes in terms of $p$ is often exactly where the two types of bounds over $\F_p$ meet. See \cite{RRS} for more discussion along these lines.

The same can be said about our second application of Theorem \ref{mish}, Theorem \ref{erd}. It yields a new result for the Erd\H os distance problem in three dimensions in positive characteristic, which is not too far off what is known over the reals. A set $S$ of $N$ points not supported in a single semi-isotropic plane in $\F^3$, contains a point, from which $\Omega\left[\min\left(\sqrt{N},p\right)\right]$ distances are realised. Semi-isotropic planes are planes spanned by two mutually orthogonal vectors $\ee_1,\ee_2$, such that $\ee_1\cdot \ee_1=0$, while $\ee_2\cdot \ee_2\neq 0$. They always exist in positive characteristic -- see \cite{HI} for explicit constructions in finite fields -- and one can have point sets with very few distinct distances within these planes.

We mention in passing another application of Theorem \ref{mish}, which is  Corollary \ref{intersections} appearing midway through the paper, concerning the prime residue field $\F_p$. Given {\em any} family of $\Omega(p^2)$ straight lines in $G={SL}_2(\F_p)$, their union takes up a positive proportion of $G$. In Lie group-theoretical terminology these lines are known as generalised horocycles, that is right cosets of one-dimensional subgroups conjugate to one of the two one-dimensional subgroups of triangular matrices with $1$'s on the main diagonal. (See, e.g., \cite{BM} as a general reference.) A similar claim in  $\F_p^3$ is  false, for all the lines may lie in a small number of  planes. Nonetheless our Corollary \ref{intersections} is not new and follows from a result of Ellenberg and Hablicsek \cite{EH}. They extend to $\F_p^3$ another, earlier relative to Theorem \ref{gkt}, algebraic theorem of Guth and Katz over $\C$, from another breakthrough paper \cite{GKprime}. The  assumption required for that in \cite{EH} is that all planes be relatively ``poor''.

\begin{remark} The presence of the characteristic $p$ in the constraint of Theorem \ref{mish} and its applications makes a positive characteristic field $\F$ somewhat morally  just $\F_p$, for Theorem \ref{Salmon} is not true otherwise. Replacing this constraint by more elaborate ones in the context of finite extensions of $\F_p$ may be possible for $p>2$ provided that classification of exceptional cases as to Salmon's theorem becomes available. Voloch \cite{V} conjectures that  an irreducible flechnodal surface of degree $d$ may be unruled only if $p$ divides $d(d-1)(d-2)$ and gives  evidence in this direction. See also \cite{EH} for examples of such {\em flexy} surfaces and discussion from the incidence theory viewpoint.\label{rem} \end{remark}

\medskip
Let us give an outline of the proof of Theorem \ref{mish} to motivate the forthcoming background material in Section \ref{setup}. First off, Theorem \ref{gkt} needs to be extended to the case of pair-wise intersections between two families of $m$ and $n$ lines, respectively. The only way to do so to meet our purpose, in view of Remark \ref{sharps}, is the cheap one. If $m$ is much bigger than $n$,  partition the $m$ lines into $\sim \frac{m}{n}$ groups of  $\sim n$ lines each and apply a generalisation to $\F$ of Theorem \ref{gkt} separately to count incidences of each group with the family of $n$ lines. 

Let us proceed assuming $m=n$. Let $q\in P,\,\pi\in\Pi$ be a point and a plane in $\Pro^3,$ and $q\in\pi$. Draw in the plane $\pi$ all lines, incident to the point $q$. In line geometry literature this figure is called a {\em plane pencil} of lines. It is represented by a line in the space of lines, that is the Klein quadric ${\mathcal K}$, a four-dimensional hyperbolic projective quadric in $\Pro^5$, whose points are in one-to-one correspondence with lines in $\Pro^3$ via the so-called {\em Klein map}. If the characteristic of $\F$, $p\neq 2$, the line pencil gets represented in $\Ka$ as follows. The Klein image of the family of all lines incident to  $q$ is  a copy of $\Pro^2$ contained in $\Ka$, a so-called $\alpha$-plane. The family of  all lines contained in $\pi$ is also represented by a copy of $\Pro^2$ contained in $\Ka$, a so-called $\beta$-plane. A pair of planes of two distinct types in ${\mathcal K}$ typically do not meet. If they do, this happens along a copy of $\Pro^1,$ a line in ${\mathcal K}$, which is the Klein image of the above line pencil, if and only if $q\in\pi$.

Thus the number of incidences $|I(P,\Pi)|$ equals the number of lines along which the corresponding sets of $\alpha$ and $\beta$-planes meet in ${\mathcal K}$. One can now restrict the arrangement of planes in ${\mathcal K}$ from $\Pro^5$ to a generic hyperplane $\Pro^4$ intersecting $\Ka$ transversely. Its intersection with $\Ka$ is a three-dimensional sub-quadric $\Ge$, whose pre-image under the Klein map is called a {\em regular  line complex.} There is a lot of freedom in choosing the generic subspace $\Pro^4$ to cut out $\Ge$. Or, one can fix the subspace $\Pro^4$ in the ``phase space'' $\Pro^5$ and realise this freedom to allow for certain projective transformations of the ``physical space'' $\Pro^3$ and its dual.  For there is a one-to-one correspondence between regular line complexes and so-called {\em null polarities} -- transformations from $\Pro^3$ to its dual by non-degenerate skew-symmetric matrices. See \cite{PW}, Chapter 3 for general theory. 

The benefit of having gone down in dimension from $\Ka$ to  $\Ge$ is that $\alpha$ and $\beta$-planes restrict to $\Ge$ as lines, which may generically meet  only if they are of different type. This is because two planes of the same type intersect at one and only one point in $\Ka$. So one can choose the subspace $\Pro^4$, defining $\Ge$ in such a way that it contains none of the above finite number of points. If the field $\F$ is finite, the latter finite set may appear to be sizable in comparison with the size of $\Ge$ itself. However, just like in the proof of Theorem \ref{gkt} one works in the algebraic closure of  $\F$, which is infinite.  Thus the only place where the characteristic $p$ of $\F$ makes a difference is within the body of the Guth-Katz theorem to ensure the validity of Salmon's Theorem \ref{Salmon}. The corresponding constraint in terms of $p$ is stated explicitly and with constants in \cite{Ko}, Corollary 40.

Having restricted the $\alpha$ and $\beta$-planes as lines in $\Ge$ we end up with two families of lines there, such that lines of the same type do not meet. The number of incidences $|I(P,\Pi)|$ equals the number of pair-wise intersections of these lines. The lines satisfy the input conditions of Theorem \ref{gkt}, the only difference being that they live in the three-quadric $\Ge\subset\Pro^4$, rather than $\Pro^3$. But one can always project a finite family of lines from higher dimension to $\Pro^3$ so that skew lines remain skew. Thereupon we find ourselves in $\Pro^3$, and what's left for the proof of Theorem \ref{mish} has been essentially worked out by Guth-Katz and Koll\'ar. This seems a bit like a waste, for the space of lines in $\Ge$ is three, rather than four-dimensional. Yet we could not conceive a better theorem for $\Ge$, but for a chance of slightly better constants. Plus, Remark \ref{sharps} suggests that a stronger theorem about $\Ge$ must have more restrictive assumptions than Theorem \ref{mish}.

\medskip
To conclude this section, we briefly summarise the key steps in the beautiful proof of Guth and Katz, which we retell with small modifications as the proof of Theorem \ref{gkt2} in the main body of the paper. We could have almost got away with just citing \cite{Ko}, Sections 3 and 4 but for a few extra details, since we still need to bring the collinearity parameter $k$ into play.

 Assuming that there are some $C n^{\frac{3}{2}}$ pair-wise line intersections in $\Pro^3$ enables one to put all the lines, supporting more than roughly the average number of incidences per line, on a polynomial surface $Z$ of degree $d\sim \frac{ \sqrt{n} }{C},$ so most of the incidences come from within factors of $Z$. One can use induction in $n$ to effectively assume that the number of these lines is $\Omega(n)$. Then Salmon's theorem implies that $Z$ must have a ruled component, containing a vast majority of the latter lines. One should not bother about non-ruled factors by the induction hypothesis. However, a non-cone ruled factor of degree $d>2$ can only support a relatively small number of incidences. Since lines of the same type do not meet, having many incidences within planes or cones is not an option either.
 
Hence if there are $C n^{\frac{3}{2}}$ incidences, $Z\subset \Pro^3$ must have a doubly-ruled quadric factor, containing many lines. Once we lift $Z$ back to $\Ge$, this means having many lines of each type in the intersection of $\Ge\subset\Pro^4$ with a $\Pro^3$. Finally, an easy argument in the forthcoming Lemma \ref{lem} shows that intersections of $\Ge$ with a $\Pro^3$ can be put into correspondence with what happens within the original arrangement of points and planes in the ``physical space''. Namely lines of the two types meeting in $\Ge\cap\Pro^3$ represent precisely incidences of the original points and planes along some line in the $\Pro^3$. This brings the collinearity parameter $k$ into the incidence estimate and completes the proof.

\section{Acknowledgment}

The author is grateful to Jon Selig for educating him about the Klein quadric. Special thanks to J\'ozsef Solymosi for being the first one to point out a mistake in the estimate of Theorem \ref{mish} in the original version of the paper and two anonymous Referees for their patience and attention to detail.  This research was conceived in part when the author was visiting the Institute for Pure and Applied Mathematics (IPAM), which is supported by the National Science Foundation.

\section{Geometric set-up}
\subsection{Background} \label{setup}
We begin with a brief introduction of the Klein, alias Klein-Pl\"ucker quadric ${\mathcal K}$. See \cite{PW}, Chapter 2 or \cite{JS}, Chapter 6 for a more thorough treatment.

The space of lines in $\Pro^3$ is represented as a projective quadric, known as  the {\em Klein quadric} $\mathcal K$ in $\Pro^5$, with projective coordinates $(P_{01}:P_{02}:P_{03}:P_{23}:P_{31}:P_{12})$, known as {\em Pl\"ucker coordinates.} The latter {\em Pl\"ucker vector} yields the {\em Klein image} of a line $l$ defined by a pair of points $q=(q_0:q_1:q_2:q_3)$ and $u=(u_0:u_1:u_2:u_3)$ in $\Pro^3$ that it contains, under the {\em Klein map}, defined as follows:
\begin{equation}
P_{ij}=q_iu_j-q_ju_i,\qquad i,j=0,\ldots,3.
\label{Pc}\end{equation}
It is easy to verify that once $\{P_{ij}\}$ are viewed as homogeneous coordinates, this definition does not depend on the particular choice of the pair of points on the ``physical line'' $l$, and there are $6=4\cdot 3/2$ independent projective  Pl\"ucker coordinates $P_{ij}$. We use the capital $L\in \Pro^5$ for the Pl\"ucker vector, which is the Klein image of the line $l\subset \Pro^3$.

For an affine line in $\mathbb{F}^3$, obtained by setting $q_0=u_0=1$, the  Pl\"ucker coordinates acquire the meaning of a projective pair of three-vectors $(\om: \ve)$, where $\om=(P_{01},P_{02},P_{03})$ is a vector in the direction of the line and for any point $\qu=(q_1,q_2,q_3)$  on the line, $\ve =  (P_{23},P_{31},P_{12}) = \qu\times\om$ is the line's moment vector\footnote{In this section we use boldface notation for three-vectors.  The essentially Euclidean vector product notation is to keep the exposition as elementary as possible: in  $\mathbb{F}^3$ the notation  $\ve = \qu\times\om$ means only that $\ve$ arises from $\om$ after multiplication on the left by the skew-symmetric matrix $T=ad({\qu})$, with $T_{12}=-q_3,\,T_{13}=q_2,\,T_{23}=-q_1$.\label{ads}}  with respect to the fixed origin. Lines in the plane at infinity have $\om=\boldsymbol 0$. We use the boldface notation for three-vectors throughout.

Conversely, one can denote $\om=(P_{01},P_{02},P_{03}),\; \ve=(P_{23},P_{31},P_{12}),$  the Pl\"ucker coordinates then become $(\om:\ve)$, and treat $\om$ and $\ve$ as vectors in $\mathbb{F}^3$, bearing in mind that as a pair they are projective quantities. The equation of the Klein quadric ${\mathcal K}$ in $\Pro^5$ is
\begin{equation}
P_{01}P_{23}+P_{02}P_{31}+P_{03}P_{12}=0,\;\mbox{ i.e., }\; \om\cdot\ve=0.
\label{Klein}\end{equation}
More formally, equation \eqref{Klein} arises after writing out, with the notations \eqref{Pc}, the truism
$$
\det\left(\begin{array}{cccccc} q_0&u_0&q_0&u_0\\q_1&u_1&q_1&u_1\\q_2&u_2&q_2&u_2\\q_3&u_3&q_3&u_3\end{array}\right) = 0.
$$
Two  lines  $l,l'$ in  $\Pro^3$, with Klein images $$L=(P_{01}:P_{02}:P_{03}:P_{23}:P_{31}:P_{12}),\qquad L'=(P'_{01}:P'_{02}:P'_{03}:P'_{23}:P'_{31}:P'_{12})$$ meet  in  $\Pro^3$ if and only if
\begin{equation}\label{intersection}
P_{01}P'_{23} +  P_{02}P'_{31} + P_{03}P'_{12} + P'_{01}P_{23} +  P'_{02}P_{31} + P'_{03}P_{12}\;=\;0.\end{equation}

The left-hand side of \eqref{intersection} is called the {\em reciprocal product} of two Pl\"ucker vectors. If they are viewed as $L=(\om:\ve)$ and  $L'=(\om':\ve')$, the intersection condition becomes
\begin{equation}\label{intersectionv}
\om\cdot \ve'+ \ve\cdot \om' = 0.\end{equation}
Condition \eqref{intersection} can be restated as 
\begin{equation}
L^T\mathcal QL'  = 0,\qquad
\mathcal Q = \left(\begin{array}{ccc} 0 & I_3\\ I_3 & 0\end{array}\right),\label{qum}\end{equation}
where $I_3$ is the $3\times 3$ identity matrix. 

It is easy to see by \eqref{intersection}, after taking the gradient in \eqref{Klein} that a hyperplane $\Pro^4$ in $\Pro^5$ is tangent to $\Ka$ at some point $L$ if and only if the covector defining the hyperplane is itself in the Klein quadric in the dual space. Moreover, it follows from (\ref{intersection}) that  the intersection  of $\Ka$ with the tangent hyperplane $T_L \Ka\cap \Ka$ through $L$ consists of $L'\in \Ka$, which are the Klein images of all lines $l'$ in $\Pro^3$, incident to the line $l$, represented by $L$. The union of all these lines $l'$ is called a {\em singular line complex.}

\subsubsection{Two rulings by planes and line complexes}
The largest dimension of a projective subspace contained in $\Ka$ is two.  This can be seen as follows. After the coordinate change $\boldsymbol x = \om-\ve$, $\boldsymbol y = \om+\ve$, the equation \eqref{Klein} becomes
\begin{equation}
\|\boldsymbol x\|^2= \|\boldsymbol y\|^2.
\label{xform}\end{equation}
This equation cannot be satisfied by a  $\Pro^3$. It can be satisfied by a $\Pro^2$ if and only if $\boldsymbol y= M\boldsymbol x,$ for some orthogonal matrix $M$.  We further assume that ${\rm char}(\mathbb F)\neq 2$, which is crucial. For then there are two cases, corresponding to $\det M=\pm 1$. The two cases correspond to  two ``rulings'' of  $\Ka$ by planes, which lie entirely in it, the fibre space of each ruling being $\Pro^3$.

To characterise the two rulings, called  $\alpha$ and $\beta$-planes, corresponding to $\det M=\pm1$, respectively,  one returns to the original coordinates $(\om:\ve)$. After a brief calculation, see \cite{JS}, Section 6.3, it turns out  that
Pl\"ucker vectors in a single $\alpha$-plane in $\Ka$ are Klein images of lines in $\Pro^3$, which are concurrent at some point $(q_0:q_1:q_2:q_3)\in \Pro^3$. If  the concurrency point is $(1:\qu)$, which is identified with $\qu\in \mathbb F^3$,  the $\alpha$-plane is a graph $\ve = \qu\times \om$. Otherwise, an ideal concurrency point  $(0:\om)$ gets identified with some fixed $\om$, viewed as a  projective vector. The corresponding $\alpha$-plane is the union of the set of parallel lines in $\mathbb{F}^3$ in the direction of $\om$, with Pl\"ucker coordinates $(\om:\ve)$, so $\ve\cdot\om=0,$ by \eqref{Klein}, and the set of lines in the plane at infinity incident to the ideal point $(0:\om)$. The latter lines have Pl\"ucker coordinates $(\boldsymbol 0:\ve),$ with once again $\ve\cdot\om=0$.

Similarly, Pl\"ucker vectors lying in a $\beta$-plane represent co-planar lines in $\Pro^3$. A ``generic'' $\beta$-plane is a graph $\om = \ur\times \ve$, for some $\ur \in \mathbb F^3$. The case $\ur =\boldsymbol 0$ corresponds to the plane at infinity, otherwise the equation of the co-planarity plane in $\mathbb{F}^3$ becomes
\begin{equation}\ur\cdot \qu =-1.\label{refer}\end{equation}
If $\ur$ gets replaced by a fixed ideal  point $(0:\ve)$, the corresponding $\beta$-plane comprises lines, coplanar in planes through the origin: $\ve\cdot \qu = 0$. The $\beta$-plane in the Klein quadric is formed by the set of lines with Pl\"ucker coordinates $(\om:\ve)$, plus the set of lines through the origin in the co-planarity plane. The latter lines have Pl\"ucker coordinates $(\om:\boldsymbol 0)$. In both cases one requires $\om\cdot\ve = 0$.

Two planes of the same ruling of $\Ka$ always meet at a point, which is the line defined by the two concurrency points in the case of $\alpha$-planes. An $\alpha$ and a $\beta$-plane typically do not meet. If they do -- this means that the concurrency point $q$, defining the $\alpha$-plane lives in the plane $\pi$, defining the $\beta$-plane.
The intersection is then a line, a copy of $\Pro^1$ in ${\mathcal K}$, representing a {\em plane pencil of lines}. These are the lines in $\Pro^3$, which are co-planar in $\pi$ and concurrent at $q$. Conversely, each line in ${\mathcal K}$ identifies the pair ($\alpha$-plane, $\beta$-plane), that is the plane pencil of lines uniquely. Moreover points $L,L'\in \Ka$ can be connected by a straight line in $\Ka$ if and only if the corresponding lines $l,l'$ in $\Pro^3$ meet, cf. \eqref{qum}.

From non-degeneracy of the reciprocal product it follows that the reciprocal-orthogonal projective subspace to a $\alpha$ or $\beta$-plane is the plane itself. Hence, a hyperplane in $\Pro^5$ contains a $\alpha$ or $\beta$-plane if and only if it is a $T_L(\Ka)$ at some point $L$, lying in the plane.

It follows that a singular line complex arises if and only if the equation of the hyperplane intersecting $\Ka$ is $(\boldsymbol u:\boldsymbol w)^T(\boldsymbol \omega:\boldsymbol v)=0$, with the dual vector
$(\boldsymbol u:\boldsymbol w)$ itself such that $\boldsymbol u\cdot\boldsymbol w=0$. Otherwise the Klein pre-image of the intersection of the hyperplane with $\Ka$ is called a  regular line complex.

We remark that a geometric characterisation of a regular line complex is that it is a set of invariant lines of some {\em null polarity}, that is a projective map from $\Pro^3$ to its dual $\Pro^{3*}$ defined via a $4\times 4$ non-degenerate skew-symmetric matrix. In particular, a null polarity assigns to each point $q\in \Pro^3$ a plane $\pi(q)$, such that $q\in \pi$. See \cite{PW}, Chapter 3 for more detail.

A particular example of the Klein image of a regular line complex arises if one sets $\omega_3=v_3$, i.e. $x_3=0$ in coordinates \eqref{xform}. One can identify $(-x_1:x_2:0:y_1:y_2:1)$ with $\mathbb F^4$, getting
$$
x_1y_1 -  x_2y_2 = 1
$$
for the affine part of $\Ge$, which can be identified with the group $SL_2(\mathbb F)$.

The following lemma describes the intersection of a regular complex with a singular one. 
\begin{lemma}
\label{lem} Let $l$ be a line in $\Pro^3$, represented by $L\in\Ka$. Then   $\Ka\cap T_L\Ka$ contains $\alpha$ and $\beta$-planes, corresponding, respectively to points on $l$ and planes containing $l$.  Given two hyperplanes $S_1$, $S_2$  in $\Pro^5$, suppose  $\Ka\cap S_1$ is the Klein image of a regular line complex. Consider the intersection $\Ka\cap S_1\cap S_2$. If the field $\F$ is algebraically closed, $\Ka\cap S_1\cap S_2 = \Ka\cap S_1\cap S'_2$, where $S'_2$ is tangent to $\Ka$ at some point $L$.  That is, $\Ka\cap S_2'$ is the Klein image of the singular line complex of lines in $\Pro^3$ meeting the Klein pre-image $l$ of $L$. \end{lemma}

\begin{proof} The first statement follows immediately by definitions above. 
To prove the second statement, suppose $S_2$ is not tangent to $\Ka$. Let the two line complexes be defined by dual vectors $(\boldsymbol u:\boldsymbol w)$ and $(\boldsymbol u':\boldsymbol w')$. If $\F$ is algebraically closed, the line $t_1(\boldsymbol u:\boldsymbol w) + t_2 (\boldsymbol u':\boldsymbol w')$ in the dual space will then intersect the Klein quadric in the dual space, a point of intersection $L$ defining $S_2'$.  

Note, however, that if $S_2$ is itself tangent to $\Ka$ at $L$, then there is only one solution, $L$ itself, otherwise there are two.
\end{proof}

\subsubsection{Reguli} For completeness purposes and since reguli appear in the formulation of Theorem \ref{gkt} we give a brief account in this section. See also the next section on ruled surfaces.

The $\alpha$ and $\beta$-planes represent a degenerate case when a subspace $S= \Pro^2$ of $\Pro^5$ is contained in $\Ka$.  Assume that $\mathbb{F}$ is algebraically closed, then any $S$ intersects $\Ka$. The non-degenerate situation would be $S$ intersecting $\Ka$ along a irreducible conic curve. This  curve in $\Ka$ is called a {\em regulus}, and the union of lines corresponding to in in the physical space forms a single ruling of a doubly-ruled quadric surface. One uses the term regulus to refer to both the above curve in $\Ka$ and the family of lines in $\Pro^3$ this curve represents.

Choose affine coordinates, so that the equations of the two-plane $S$ can be written as
$$
A\om + B\ve = \boldsymbol 0,
$$
where $A,B$ are some $3\times3$ matrices. For points in $S\cap\Ka$, which do not represent lines in the plane at infinity in $\Pro^3$, we can write $\ve =\qu\times \om$, where $\qu$ is some point in $\mathbb{F}^3$, on  the line  with Pl\"ucker coordinates $(\om:\ve)$, and $\om\neq \boldsymbol 0$. If $T$ denotes the skew-symmetric matrix $ad(\qu)$ we obtain
$$
(A-BT)\om =\boldsymbol 0\qquad \Rightarrow\qquad \det(A-BT)=0.
$$
This a quadratic equation in $\qu$, since $T$ is a $3\times 3$ skew-symmetric matrix, so $\det T=0$.  If the above equation has a  linear factor in $\qu$, defining a plane  $\pi\subset\Pro^3$, then $S\cap \Ka$ contains a line, which represents a pencil of lines in $\pi$.  If the above quadratic polynomial in $\qu$ is irreducible, and  $\mathbb{F}$ is algebraically closed, one always gets a quadric irreducible surface in $\Pro^3$ as the union of lines in the regulus, see Lemma \ref{rs} in the next section.

In the latter case, by Lemma \ref{lem}, the two-plane $S$  in $\Pro^5$ can be obtained as the intersection of three four-planes, tangent to $\Ka$ at some three points $L_1,L_2,L_3$, corresponding to three mutually skew lines in $\Pro^3$. Thus a regulus can be redefined as the set of all lines in $\Pro^3$, meeting three given mutually skew lines $l_1,l_2,l_3$. Its Klein image is a conic.

Each regulus has a reciprocal one, the Klein image of the union of all lines incident to any three lines, represented in the former regulus. These lines form the second ruling of the same quadric doubly-ruled surface. See \cite{JS}, Section 6.5.1 for coordinate description of reciprocal reguli. \label{rgl}

\subsubsection{Algebraic ruled surfaces} \label{ruled} Differential geometry of ruled surfaces is a rich and classical field of study.  From a historical perspective, 
it was Pl\"ucker who pretty much invented the subject in the two-volume treatise \cite{Pl}, which was completed after his death by Klein. 

We give the minimum background on algebraic ruled surfaces in $\Pro^3$. In this whole section  the field $\F$ is assumed to be algebraically closed, of characteristic $p\neq 2$. See \cite{PW}, Chapter 5 for the discussion in the case $p=0$. 
In positive characteristic the basics of algebraic theory of ruled surfaces are in many respects the same, and for our modest designs we need only these basics. 

A ruled surface is defined as a smooth projective surface over an algebraically closed field that is birationally equivalent to a surface $\Pro\times \mathcal C$ where $\mathcal C$ is a smooth projective curve of genus $g\geq0$. See, e.g. \cite{Ba}, \cite{Li} for general theory of algebraic surfaces. 
Also Koll\'ar  (see \cite{Ko}, Section 7)  presents in terms of more formal  algebraic geometry a brief account of facts about ruled surfaces, necessary for the proof of Theorem \ref{gkt}. Since he only mentions the Klein quadric implicitly through a citation we review these facts below.  

Informally, an algebraic ruled surface  is a surface in $\mathbb P^3$ composed of a polynomial family of lines. We assume the viewpoint from Chapter 5 of the book by  Pottmann and Wallner \cite{PW}, where an algebraic ruled surface is identified with  a polynomial curve $\Gamma$ in the Klein quadric. The union of  lines, Klein pre-images of the points of $\Gamma$ draws a surface $Z\subset \mathbb P^3$ called the {\em point set} of $\Gamma$. It is easy to show that $Z$ is then an algebraic surface, that is a projective variety of dimension $2$. A line in $Z$, which is the Klein pre-image of a point of $\Gamma$ is called a {\em generator.} A regular generator $L$, that is a regular point of $\Gamma\subset \mathbb K$ is called {\em torsal} in the special case when the tangent vector to $\Gamma$ at $L$ is also in $\Ka$. The Klein pre-image of a regular torsal generator necessarily supports a singular surface point, called {\em cuspidal point}.
An irreducible component of $\Gamma$ is referred to as a {\em ruling} of $\Gamma$. The same term {\em ruling} is applied to the corresponding family of lines, ruling the surface $Z$.

Here is s basic genericity statement about ruled surfaces. See, e.g.,  \cite{PW}, Chapter 5.

\begin{lemma}\label{rs} Let $\Gamma$ be an algebraic curve in $\Ka$, with no irreducible component contained in the intersection of $\Ka$ with any $\Pro^2$. Let $Z$ be the point set of $\Gamma$.  The subset of $Z$, which is the union of all pair-wise intersections of different rulings of $\Gamma$ and all cuspidal points is a subset of the set of singular points of $Z$. It is contained in an algebraic subvariety of dimension $\leq 1$. 

Besides, the curve $\Gamma$ is irreducible if and only if its point set $Z$ is irreducible.
 \end{lemma}

We do not give a proof but for a few remarks. The conditions of Lemma 
\ref{rs} rule out the cases when $Z$ has a plane or smooth quadric component. Clearly, a plane can be the point set for many rulings of lines lying therein, a smooth quadric has two reciprocal reguli, and is therefore an example when the union of the two reguli, not irreducible as a ruled surface has an irreducible point set.

Let $Z$ further denote the point set of a ruling. Suppose, $Z$ contains three lines $l_{1}, l_{2}, l_3$ incident to every line in the ruling. If, say $l_{1}$ and $l_{2}$ meet, then $Z$ is either a plane, and hence the ruling lies in an $\alpha$ or  $\beta$-plane, depending on whether or not $l_3$ also meets $l_1$ and $l_2$ at the same point. If the three lines are mutually skew, then the ruling is contained in the intersection of three singular line complexes $T_1,T_2,T_3$, corresponding to the three lines. Their intersection is represented in $\Ka$ as the latter's transverse section by a $\Pro^2$ along a conic, that is a regulus. Then $Z$ is a irreducible quadric surface, which has a reciprocal ruling: the set of lines incident to any three lines in the former ruling. See the above discussion of reguli, as well as \cite{JS}, Chapter 6 for more details.

Conversely, if a ruling is contained in a $\alpha$-plane, then $Z$ is a cone: all the generators are incident at the concurrency point defining the $\alpha$-plane. It a ruling lies in a $\beta$-plane, then $Z$ is a plane. If the ruling arises as a result of transverse intersection of a $\Pro^2$ with $\Ka$, it is either a pencil of lines or a regulus. In the former case $Z$ is a plane, in the latter case an irreducible doubly-ruled quadric. 

\medskip
An important part of the proof of Theorem \ref{gkt} is the  claim that one cannot have too many line-line incidences within a higher degree irreducible ruled surface, which is not a cone. It is essentially the rest of this section that is directly relevant to Theorems \ref{gkt} and \ref{mish}.

\begin{lemma}\label{rss} Let $\Gamma$ be an algebraic ruled surface of degree $d$, whose point set $Z$ has no plane component. Then the degree of $Z$ equals $d$. A generator in a ruled surface of degree $d$, which does not have a cone component, meets at most $d-2$ other generators.
\end{lemma}

\begin{proof} By the preceding argument, the theorem is true for $d=2$, so one may assume that conditions of Lemma \ref{rs} are satisfied. 
Since $\F$ is algebraically closed, a generic line $l$ in $\Pro^3$ intersects $Z$ exactly $d$ times at points meeting one generator each. It follows that for the Klein image $L$ of $l$, one has 
$$
L^T \mathcal Q L' = 0,
$$
for $d$ distinct $L'\in \Gamma$. Thus the curve $\Gamma$ meets a hyperplane $T_L\Ka$ in $\Pro^5$ transversely $d$ times, and hence has degree $d$.

If in the latter equation $L$ no longer represents a generic line in $\Pro^3$ but a generator of $\Gamma$, and the above equation must still have $d$ solutions, counting multiplicities. Besides $L'=L$ has multiplicity at least $2$, since the intersection of $\Gamma$ with $T_L\Ka$ at $L$ is not transverse.
\end{proof}
It also follows that the point set of an irreducible ruled surface $\Gamma$  of degree $d\geq3$ cannot be a smooth projective surface. The point set of $\Gamma$  will necessarily have  singular points where two generators meet or a cuspidal points of  torsal generators.

It is also well known that the point set of an irreducible ruled surface of degree $d\geq3$ can support at most two non-generator {\em special lines} which intersect each generator. This is because special lines must be skew to each other, or  one has a plane. But then if there are three or more special lines, one has a quadric.

\subsection{Point-plane incidences in $\Pro^3$ are line incidences in a three-quadric in $\Pro^4$} \label{geom}
We can now start moving towards Theorem \ref{mish}. Assume that $\F$ is algebraically closed or pass to the algebraic closure still calling it $\F$. It is crucial for this section that $\F$ not have characteristic $2$. Let  $\Ka\subset\Pro^5$ be the Klein quadric, $\Ge=\Ka\cap S$, for a four-hyperplane $S$ whose defining covector is not in the Klein quadric in $\Pro^{5*}$. E.g., $\Ge$ may be defined by the equation $P_{03}=P_{12}$. Since $\Ge$ contains no planes, each $\alpha$ or $\beta$-plane in ${\mathcal K}$ intersects ${\mathcal G}$ along a  line. We therefore have two line families $L_\alpha,L_\beta$  in $\Ge$. We warn the reader from confusing lines lying in the three-quadric $\Ge\subset \Pro^4\subset\Ka\subset\Pro^5$ in the ``phase space'' with lines from the regular line complex in the ``physical space'' $\Pro^3$ that $\Ge$ is the Klein image of.

The following lemma states that one can assume $L_\alpha\cap L_\beta=\emptyset,$ as well as that the lines within each family do not meet each other.

\begin{lemma} \label{tog} Suppose, $\F$ is algebraically closed and not of characteristic $2$. To every finite point-plane arrangement $(P,\Pi)$ in $\Pro^3$  one can associate two distinct families of lines $L_\alpha,L_\beta$ contained in some three quadric $\Ge=\Ka\cap S$, where the four-hyperplane $S$ is not tangent to $\Ka$, with the following property. No two lines of the same family meet; $|L_\alpha|=m$, $|L_\beta|=n$, and $|I(P,\Pi)| = |I(L_\alpha,L_\beta)|$, where $I(L_\alpha,L_\beta)$ is the set of pair-wise incidences between the lines in $L_\alpha$ and $L_\beta$.

Alternatively, one can regard $S$ as fixed and find a new point-plane arrangement $(P',\Pi')$ in $\Pro^3$ with the same $m,n$ and the number of incidences, to which  the above claim applies.

Besides, if $k_m,k_n$ are the maximum numbers of, respectively, collinear points and planes in $P,\Pi$, they are now the maximum numbers of lines in the families $L_\alpha,L_\beta$, respectively, contained in the intersection of $\Ge\subset S$ with a projective three-subspace in $S$.
\end{lemma}

\begin{proof}
Suppose, we have an incidence $(p,\pi)\in P\times\Pi$. This means that the $\alpha$-plane defined by $q\in P$ and the $\beta$-plane
defined by $\pi\in \Pi$ intersect along a line in ${\mathcal K}$. There are at most $m^2+n^2$ points in ${\Ka}$ where planes of the same type meet and at most $mn$ lines along which the planes of different type may possibly intersect.

We must choose $\Ge$ that is a hyperplane $S$ in $\Pro^5$ intersecting $\Ka$ transversely, so that it supports none of the above lines or points in $\Ka$. This means avoiding a finite number of linear constraints on the dual vector $U^T\in \Pro^{5*}$ defining $S$. Since $\F$ is algebraically closed, it is infinite, and such $S$ always exists, for $m,n$ are finite. The covector $U^T$ defining $S$ must (i) not lie in the Klein quadric in $\Pro^{5*},$ and (ii) be such that $U^T L_i\neq 0$ for at most $m^2+n^2 + mn$ Pl\"ucker vectors $L_i$.  There is a nonempty Zariski open set of such covectors in $\Pro^{5*}$.

To justify the second claim of the lemma we use the fact that there is one-to-one correspondence between so-called null polarities and regular line complexes. A null polarity is a projective transformation from $\Pro^3$ to its dual, given by a non-degenerate $4\times 4$ skew-symmetric matrix.  The six above-diagonal entries of the matrix are in one-to-one correspondence with the covector defining the regular line complex. The fact that the skew-symmetric matrix is non-degenerate is precisely that the covector not lie in the Klein quadric. See \cite{PW}, Chapter 3 for general theory of line complexes.

Hence the following procedure is equivalent to the above-described one of choosing the transverse hyperplane $S$ defining $\Ge$. Fix $S$ and find a null polarity, whose application  to the original arrangement of planes and points in $\Pro^3$ yields a new point-plane arrangement as follows. The roles of points and planes get reversed, and we now have the set of $m$ planes $\Pi'$ and the set of $n$ points $P'$, with the same number of incidences $|I(P,\Pi)|$. Take a dual arrangement so points become again points and planes are planes. However, no two lines of the same type, arising in $\Ge\subset S$ after the procedure described in the beginning of this section applied to the arrangement  $(P',\Pi')$, will intersect. 

The last claim of Lemma \ref{tog} follows from Lemma \ref{lem}.
\end{proof}

Fixing the transverse hyperplane $S$ may be interesting for applications, when the affine part of the quadric $\Ge$ becomes, say the Lie group ${SL}_2(\F)$, with its standard embedding in $\F^4$. Suppose there are $n$ lines supported in a fixed $\Ge$. Each line in $\Ge$ is a line in $\Ka$ and therefore corresponds to a unique plane pencil of lines in the ``physical space" $\Pro^3$, that is a unique pair $\alpha$ and $\beta$-plane intersecting along this line. I.e., there is a unique pair $(q,\pi(q))$, where the point $q$ lies in the plane $\pi(q)$. (Conversely,  $\Ge$ viewed as a null polarity is defined by the linear skew-symmetric linear map $q\to\pi(q),$ see \cite{PW}, Chapter 3.) Hence, given a family of $n$ lines in $\Ge$, the problem of counting their pair-wise intersections  can be expressed as counting the number of incidences in $I(P,\Pi)$, where $P=\{q\}$ and $\Pi=\{\pi(q)\}$. Moreover, $|P|,|\Pi|=n$, for two different planes of the same type will never intersect $\Ge$ along the same line (that is a null polarity is an isomorphism). Besides, if $k$ was the maximum number of lines in the intersection of $\Ge\subset \Pro^4$ with a $\Pro^3$, then the same $k$ stands for the maximum number of collinear points or planes, by Lemma \ref{lem}.

We have established the following statement.

\begin{lemma}\label{convert} Suppose, $\F$ is algebraically closed and not of characteristic $2$. Let $\mathcal L$ be a family of $n$ lines in $\Ge.$ Then there is an arrangement $(P,\Pi)$ of $n$ points and $n$ planes in $\Pro^3$, such that the number of pair-wise intersections of lines in  $\mathcal L$ equals $|I(P,\Pi)|-n$. Moreover, there are two disjoint families of $n$ new lines in $\Ge$ each, such that lines within each family are mutually skew, and the total number of incidences is $|I(P,\Pi)|-n$.  \end{lemma}

Note, the $-n$ comes from the fact that each $\pi(q)$ contains $q$.
Lemma \ref{convert} and Theorem \ref{mish} have the following corollary. This fact also follows from the results in \cite{EH} after a projection argument.  We present the proof along the lines of exposition in this section, for it  also gives an application of the formalism here.

\begin{corollary}\label{intersections} The union of any $n=\Omega(p^2)$ straight lines in $G={SL}_2(\F_p)$ has cardinality $\Omega(p^3)$, that is takes up a positive proportion of $G$.
\end{corollary}
\begin{proof} The statement is trivial for small $p$, so let $p>2$. View lines in $G\subset \F_p^4$ as lines in $\Ge\subset \Pro^4$ over the algebraic closure of $\F_p$. Pass to a point-plane incidence problem in $\Pro^3$ using Lemma \ref{convert} and then by Lemma \ref{tog} back to a line-line incidence problem in  $\Ge.$  We may change $n$ to $cn$ to make Theorem \ref{mish} applicable. The value of the absolute $c$ may be further decreased to justify subsequent steps. By the inclusion-exclusion principle one needs to show that the number of pair-wise intersections of lines is at most a fraction of $pn$. This would follow if one could apply the incidence bound \eqref{pups} with $m=n$ and, say $k=\frac{p}{2}$. 

By Lemma \ref{tog} the quantity $k$ is the maximum number of  ``new lines'' in the intersection of $\Ge$ with a projective three-hyperplane. Observe that there are more than $\frac{p}{2}$ of new lines in the intersection of $\Ge$ with a hyperplane if and only if there was the same number of  ``old lines'' in the intersection of $G$ with an affine hyperplane in $\F_p^4$.  

Let us throw away from the initial set of lines in $G$ those lines, contained in intersections of $G\subset \mathbb F_p^4$ with affine three-planes $H$, with $H\cap G$ having more than $\frac{p}{2}$ lines. Either we have a positive proportion of lines left, and no more  rich hyperplanes $H$, or we have had $\Omega(p)$ quadric surfaces  $H\cap G$ in $G$, with at least $\frac{p}{2}$ lines in each. In the former case, if $c$ is small enough, we are done by (\ref{pups}). In  the latter case, by the inclusion-exclusion principle applied within each surface, the union of lines contained therein takes up a positive proportion of each $H\cap G$, i.e., has cardinality $\Omega(p^2)$. Since $H\cap H'\cap G$, $H\neq H'$ is at most two lines, by the inclusion-exclusion principle, the union of $\Omega(p)$ of them has cardinality $\Omega(p^3)$. \end{proof}

\section{Proof of Theorem \ref{mish}}
We use Lemma \ref{tog}  to pass to the incidence problem  between two disjoint line families $L_\alpha,L_\beta$ lying in $\Ge$, now using $m= |L_\alpha|$, $n= |L_\beta|$. Lines within each family are mutually skew.

All we need on the technical side is to consider the case $m\geq n$ and adapt the strategy of the proof of Theorem \ref{gkt} to the three-quadric $\Ge$ instead of $\Pro^3$. The latter is done via a generic projection argument, and  the rest of the proof follows the outline in the opening sections. We skip some easy intermediate estimates throughout the proof, since they have been worked out accurately up to constants in \cite{Ko}, Sections 3,4.

The key issue is that any finite line arrangement over an infinite field in higher dimension can be projected into three dimensions with the same number of incidences; this fact is also stated in \cite{Ko}. Our lines lie  in  $\Pro^4$, containing the quadric $\Ge$. A pair of skew lines defines a three-hyperplane $H_i$ in $\Pro^4$. This hyperplane is projected one-to-one onto a fixed three-hyperplane $H$ if and only if the projective vector $u\in\Pro^4$ defining $H$ does not lie in $H_i$. Since we are dealing with a finite number of pairs of skew lines and $\F$ is infinite, the set of  $u$, such that the projection of the line arrangement on the corresponding three-hyperplane $H$ acts one-to-one on the set of incidences is non-empty and Zariski open.

\begin{theorem}\label{gkt2}
 Let $L_\alpha,L_\beta$ be two disjoint sets of respectively $m,n$ lines contained in the quadric $\Ge=\Ka\cap S$, where the hyperplane $S$ is not tangent to the Klein quadric $\Ka$.  Suppose, lines within each family are mutually skew. Assume that $m\geq n$, $\F$ is algebraically closed, with characteristic $p\neq 2$. Let $n\leq cp^2,$ for some absolute $c$.

Then
\begin{equation}\label{ibou}
|I(L_\alpha,L_\beta)|=O\left( m\sqrt{n}+ km\right),
\end{equation}
where $k$ is the maximum number of lines in $L_\beta$, contained in the intersection of $\Ge\subset \Pro^4$ with a subspace $\Pro^3$ in $\Pro^4$.
\end{theorem}

\begin{proof}
Following Guth and Katz, it is technically very convenient to use induction in $\min(m,n)$ and a probabilistic argument. The estimate $I=O( m\sqrt{n} )$ is true for all sufficiently small $m,n$, given  a sufficiently large $O(1)$ value $C$ of the constant in the $O$-symbol, which we fix. We do not specify how large $C$ should be, however Koll\'ar evaluates it explicitly, see \cite{Ko}. For the induction assumption to work throughout let us reset $n = \min(|L_\alpha|,|L_\beta|)$ and  $m = \max(|L_\alpha|,|L_\beta|)$.  The induction assumption will be used throughout the proof as the bound for incidences between sub-families of $(m',n')$ lines, with $n'$ sufficiently less than $n$, no matter what $m'$ is. This will enable us  to exclude from consideration the incidences that some  undesirable subsets of lines in $L_\beta$ account for, as long as they constitute a reasonably small fraction of   $L_\beta$ itself. 

Suppose, we have the smallest value of $n$, such that for some $m\geq n$  the  main  term in the right-hand side of \eqref{ibou}  fails to do the job, that is 
\begin{equation}
|I(L_\alpha,L_\beta)|  =  Cm\sqrt{n},\label{contr}\end{equation}
for some large enough constant $C$.  We will show that this assumption implies the bound $I=O(km)$, independent of $C$, which will therefore finish the proof. 

Note that since the right-hand side of the assumption \eqref{contr} is linear in $m$, it implies, by the pigeonhole principle,  that there is a subset $\tilde L_\alpha$ of $L_\alpha$ of some $\tilde m\leq m$ lines, with $\tilde m=O(n)$, such that 
$$|I(\tilde L_\alpha, L_\beta)| \geq C\tilde m\sqrt{n}.$$
We reset the notations $\tilde L_\alpha$ to $L_\alpha$ and $\tilde m$ back to $m$, but now $m=O(n)$, which is necessary for the next step.

A large proportion of incidences must be supported on  lines in $L_\alpha$, which are intersected not much less than average, say by at least $\frac{1}{4} C  \sqrt{n}$ lines from $L_\beta$ each. Let us call this popular set $L'_\alpha$. We now delete lines from $L_\beta$ randomly and independently, with probability $1-\rho$ to be chosen.  Let the random surviving subset  of $L_\beta$ be denoted as $\tilde L_\beta$. By the law of large numbers, the probability that an individual line in $L'_\alpha$ is met by  lines from  $\tilde L_\beta$ less than half the expected number of times is exponentially small in $n$, and so is $m$ times this probability, since now $m=O(n)$. Thus  there is a realisation of $\tilde L_\beta \subset L_\beta$, of size close to the expected one,
i.e., between $\frac{1}{2}\rho n$ and $2\rho n$ such that every line in $L'_\alpha$ meets at least, say \begin{equation}\label{avrg}\frac{1}{8} C \rho \sqrt{n}\end{equation} lines in $\tilde L_\beta$.

Our lines live in  $\Ge\subset\Pro^4$, with homogeneous coordinates $(x_0:\ldots:x_4)$.  By the projection argument, preceding the formulation of Theorem \ref{gkt2}, the coordinates can be chosen in such a way that lines in the union of the two families project one-to-one as lines in the $(x_1:\ldots:x_4)$-space, and skew lines remains skew.

Let  $Q$ be a nonzero homogeneous polynomial in $(x_1:\ldots:x_4)$ that vanishes on the projections of the lines in $\tilde L_\beta$ to the $(x_1:\ldots:x_4)$-space, so it will also vanish on the lines in $\tilde L_\beta$. 
The degree $d$ of $Q$ can be taken as $O\left((\rho n)^{\frac{1}{2}}\right)$.  This fact is well known, see e.g. the survey \cite{D1}. For completeness, we give a quick argument. Choose $t$ points on each of the projected lines from $\tilde L_\beta$, with or without repetitions. Let $X\subset\Pro^3$ be the corresponding set of at most $t|\tilde L_\beta|$ points. There is a nonzero homogeneous polynomial of degree $d= O[ (t|\tilde L_\beta|)^{1/3} ]$ vanishing on $X$. More precisely, it suffices to satisfy the inequality $\left(\begin{array}{c} d+3\\3\end{array}\right)>|X|$ for the degree of the polynomial. The left-hand side of the latter inequality is the dimension of the vector space of degree $d$ homogeneous polynomials in four variables; if it is  bigger than $|X|$, the evaluation map on $X$ has nontrivial kernel, by the rank-nullity theorem.

By construction of the point set $X$, the polynomial $Q$ has $t$ zeroes on each line from $\tilde L_\beta$, so in order to have it vanish identically on the union of these lines one must merely ensure that $t>d$. Hence, the above claim for $d$.

We choose the parameter $\rho$, so that the degree $d$ of $Q$ is smaller than the number of its zeroes on each line in $L_\alpha'$, which is at least \eqref{avrg}. I.e.,
$$\rho =O\left(\frac{1}{C^2}\right)<1,$$
 and thus
\begin{equation}\label{d1}d=O(\sqrt{\rho n}) =O\left( \frac{\sqrt{n}}{C}\right).
\end{equation}

Reduce $Q$ to the minimal product of irreducible factors. Denote $\bar Z$ the zero set of the polynomial $Q$ in $\Pro^3$ defined by the $(x_1:\ldots:x_4)$ variables and $\bar L'_\alpha,  \bar L'_\beta$ the projections of the corresponding line families.  Let also $Z$ denote the zero set of the polynomial $Q$ in $\Ge\subset\Pro^4$.
Recall that the projection has been chosen so that $|I(\bar L'_\alpha, \bar L'_\beta)|=|I(L'_\alpha, L'_\beta)|$ and lines in the same family still do not meet. In the sequel, when we speak of zero sets of factors of $Q$, we mean point sets in $\Pro^3$, in the $(x_1:\ldots:x_4)$ variables.

 It follows that all the lines in $\bar L'_\alpha$ are  contained in $\bar Z$, for each supports more zeroes of $Q$ than the degree $d$. For all lines from $\bar L_\beta$ that do not live in $\bar Z$, every such line will intersect $\bar Z$ at most $d$ times. The number of incidences these lines can create altogether is thus
\begin{equation}O\left( C^{-1} n^{\frac{3}{2}} \right) = O\left( C^{-1} m\sqrt{n}\right),\label{transverse}\end{equation}
which is too small in comparison with the supposedly large total number of incidences \eqref{contr}. Therefore, we may assume that, say at least $\frac{1}{2} C m\sqrt{n}$ incidences are supported on lines in $\bar L'_\alpha$ and those lines from $\bar L_\beta$ that are also contained in $\bar Z$. Suppose, the number of the latter lines is less than, say  $\frac{n}{16}$. This will contradict the induction assumption -- no matter how many lines $m'$ are there in $\bar L'_\alpha$. If $m'\geq n$, then the number of incidences, by the induction assumption, must be at most $Cm'\sqrt{n}/4$; if $m'<\frac{n}{16}$, it is at most $C n \sqrt{m'}/16<C m \sqrt{n}/16.$
Hence, there are at least $\frac{n}{16}$ lines from $\bar L_\beta$ in $\bar Z$, and we call the set of these lines $\bar L'_\beta$. To avoid taking further fractions of $n$, let us proceed assuming that $|\bar L'_\beta|=n$.

We can repeat the transverse intersection incidence counting argument for the zero set of each irreducible factor of $Q$. Suppose, the factor has degree $d'$. Then the number of incidences of lines in the zero set $\bar Z'$ of the factor with those not contained in $\bar Z'$ is at most $d' (m+n)$. Summing over the factors, we can use the right-hand side of  \eqref{transverse} as the estimate for the total over all the irreducible  factors of $Q$ number of transverse incidences. We therefore proceed assuming that there are $\Omega( C m\sqrt{n})$ of pairs of intersecting lines from the two families, each incidence occurring inside the zero set of some irreducible factor of $Q$.

Invoking Salmon's Theorem \ref{Salmon} we deduce that if  $n>11d^2-24d$, and given that $d<p$ if the characteristic $p>0$, the zero set $\bar Z$ of the polynomial $Q$ must have a ruled factor.  The latter inequality entails that almost 100\% of  lines in the $\beta$-family must lie in ruled factors. Indeed, we have $|\bar L'_\beta|=n$ lines in $Z$, and at most 
$11d^2=O(n/C^2)$ may lie in the union of non-ruled factors, provided that $d=O(\sqrt{n}/C)< p$, that is the constraint in Theorem \ref{Salmon} has been satisfied. Thus, we may not bother about what happens in non-ruled factors of $\bar Z$ by the induction assumption and proceed, having redefined $n$ slightly one more time, so that now $n$ lines from $\bar L'_\beta$ lie in ruled factors of $\bar Z$. They still have to account for $\Omega(Cm\sqrt{n})$ incidences with the lines from $\bar L'_\alpha$, for all the lines in $\bar L'_\beta$ that have been disregarded so far could only account for a small percentage of the total number of incidences.

A single ruled factor cannot be a cone, for no more than two of our lines meet at a point. However, a ruled factor of degree $d'>2$, which is not a cone, can contribute, by Lemma \ref{rss}, at most $n(d'-2)+2n+(m+n)d'$ incidences. The latter three summands come, respectively, from mutual intersections of generators, intersections of generators with special lines -- see the discussion from Lemma \ref{rss} through the end of Section \ref{ruled}  -- and intersections of lines within the factor with lines outside the factor. 

Once again, summing over irreducible ruled factors with $d'>2$, we arrive in the right-hand side term in \eqref{transverse} again -- this is too small in comparison with \eqref{contr}. Hence $Q$ must contain one or more irreducible factors $Q'$ of degree at most $2$, that is the zero set of each such $Q'$ is an irreducible doubly-ruled quadric or a plane in $\Pro^3$. If the union of these low degree factors contains only a small proportion of the lines from $\bar L'_\beta$, we once again invoke the induction assumption and contradict \eqref{contr}.

Let us reset $n$ to its original value.  The argument up to now has calmed that if \eqref{contr} is true,  we have at least $cn$ lines from $\bar L'_\beta$ lying in the union of the zero sets of low degree -- meaning degree at most two -- factors of $Q$, creating at least $cCm\sqrt{n}$ incidences with lines from $\bar L'_\alpha$ inside these factors. By the pigeonhole principle, there is a low degree factor  $Q'$, whose zero set contains at least $c\frac{n}{d}=\Omega(C\sqrt{n})$ lines from $\bar L'_\beta$.  Moreover, we can disregard whatever happens inside the union of low degree factors, each containing fewer than some $cC\sqrt{n}$ lines from  $\bar L'_\beta$, by the induction assumption.

The contribution of  plane factors of $Q$ is negligible, for each plane in $\Pro^3$ may contain only one line from each (projected) family.
Thus there is a rich degree $2$ irreducible factor $Q'$,  which defines a doubly ruled quadric surface $\bar Z'$ in  the $(x_1:\ldots:x_4)$ variables. $\bar Z'$ supports at least two lines from $\bar L'_\alpha$ in one ruling, for otherwise the total number of incidences within all such rich quadrics would be $O(C^{-1}n^{\frac{3}{2}})$. These two lines are crossed by all lines in the second ruling, that is by $\Omega(C\sqrt{n})$ lines from the family $\bar L'_\beta.$

It remains to bring the parameter $k$ in, the maximum number of lines from $L_\beta$, per intersection of $\Ge\subset \Pro^4$ with a three-hyperplane.
Let $Z'=\Ge\cap (\bar Z'\times \Pro^1)$, that is the intersection 
of the quadric $\Ge$ with the quadric, which is the zero set of $Q'$ in $\Pro^4$. Lifting lines from $\bar Z'$ to $Z'$ preserves incidences, so we arrive at the following figure in $Z'\subset \Ge$: a pair of skew lines from $L_\alpha$ crossed by $\Omega(C\sqrt{n})$  lines from $L_\beta$. The  two  lines from $L_\alpha$ determine a three-hyperplane $H$, which also contains all the $\Omega(C\sqrt{n})$ lines in question from $L_\beta$. 

By the assumption of the theorem, $H$ may contain at most $k$ lines from $L_\beta$. This means $C=O\left(\frac{k}{\sqrt{n}}\right)$. Substituting this into \eqref{contr} yields the inequality $ |I(L_\alpha,L_\beta)|=O(km).$
This completes the proof of Theorem \ref{gkt2}.
\end{proof}

Theorem \ref{gkt2} together with the preceding it discussion in Sections \ref{setup} and \ref{geom} and its outcomes stated as Lemma \ref{lem} and \ref{tog}, result straight into the claim of our main Theorem \ref{mish}.

\section{Applications of Theorem \ref{mish}}
This section has three main parts. First, we develop an application of Theorem \ref{mish} to the problem of counting vector products defined by a plane point set, extending to positive characteristic the estimates obtained over $\R$ via the Szemer\'edi-Trotter theorem. Then we use that application in a specific example  to show that in a certain parameter regime Theorem \ref{mish} is tight.  Finally, we use Theorem \ref{mish} to consider a pinned version of the Erd\H os distance problem on the number of distinct distances determined by a set of $N$ points in $\F^3$, where we also get a new bound in positive characteristic, which is not too far off the best known bound over the reals. 

Before we do this, we state a slightly stronger version of Theorem \ref{mish}, which is  more tuned for applications. The need for it comes from the fact that sometimes, when questions of geometric and arithmetic combinatorics are reformulated as incidence problems,  there are certain geometrically identifiable subsets of the incidence set that should be excluded from the count, for they correspond to some in some sense ``pathological'' scenario. We encountered this in \cite{RR}, where the Guth-Katz approach to the the  Erd\H os distance problem was applied to Minkowski  distances in the real plane. In order to get the lower bound for the number of distinct Minkowski  distances, one claims an upper bound on the number of pairs of congruent, that is equal Minkowski length line segments with endpoints in the given plane point set. However, it is easy to construct an example where the number of pairs of zero Minkowski length segments is forbiddingly large. Hence, the analysis in \cite{RR} considered only nonzero Minkowski distances, and had to elucidate how this fact gets reflected in the corresponding incidence problem for lines in three dimensions. Discounting pairs of line segments of zero Minkowski length was equivalent to discounting pair-wise line intersections within a set of specific two planes in $3D$; these planes could violate the assumption of Theorem \ref{gkt} about the maximum number of coplanar lines.

Such a restricted  application of the Guth-Katz approach was further generalised in \cite{RS}, where more $2D$ combinatorial problems have been identified, where the tandem of incidence Theorems 2.10 and 2.11 from \cite{GK} worked ``as a hammer'', if used in the restricted form, that is discounting pairwise line intersections within certain ``bad'' planes, as well as at certain ``bad'' points.   

Technically, it is Theorem 2.11 from \cite{GK}, whose restricted version required most of the work in \cite{RR}; adapting Theorem 2.10 took only a few lines of argument, and this is all that is essentially needed here regarding Theorem \ref{mish}, where we wish to discount point-plane incidences supported on a certain set of  forbidden  lines in $\Pro^3$.

 Suppose, we have a finite set of lines $L^*$ in $\Pro^3$. Define the restricted set of incidences between a point set $P$ and set of planes $\Pi$ as
\begin{equation}\label{inss}
I^*(P,\Pi) = \{(q,\pi)\in P\times \Pi: q\in \pi \mbox{ and } \forall l\in L^*, \,q\not \in l \mbox{ or } l \not\subset\pi\}.
\end{equation}

\addtocounter{theorem}{-10}

\renewcommand{\thetheorem}{\arabic{theorem}*}
\begin{theorem} \label{mishh} Let $P, \Pi$ be sets of points and planes in  $\Pro^3$, of cardinalities respectively $m,n$, with $m\geq n$. If $\F$ has positive characteristic $p$, then $p\neq 2$ and $n=O(p^2)$. For a finite set of lines $L^*$, let  $k^*$ be the maximum number of  planes, incident to any line not in $L^*$. 

Then
\begin{equation}\label{pupss}
|I^*(P,\Pi)|=O\left( m\sqrt{n}+ k^*m\right).\end{equation}
\end{theorem}
\begin{proof} We return to Section \ref{geom} to map the incidence problem between points and planes  to one between line families $L_\alpha,L_\beta$ in $\mathcal G\subset \Pro^4$.  By Lemmas \ref{lem}, \ref{tog} the set of  lines $L^*$ now displays itself as a set ${\mathcal H}^*$ of three-hyperplanes in $\Pro^4$. One comes to Theorem \ref{gkt2}, only now aiming to claim \eqref{pupss} as the estimate for the cardinality of the restricted incidence set  $I^*(L_\alpha,L_\beta)$, which discounts pair-wise line intersections within the intersections of $\Ge$ with each $h\in \mathcal H^*$,  $k^*$ replacing $k$.

The proof of Theorem \ref{gkt2} is modified as follows. Since the number of bad hyperplanes is finite, one can choose coordinates so that the intersection of each $h\in \mathcal H^*$ with $\Ge$ is defined by a quadratic polynomial $Q_h$ in $(x_1:\ldots:x_4)$. The arguments of Theorem \ref{gkt2} are copied modulo  that one assumes \eqref{contr} about the quantity $|I^*(P,\Pi)|$ and having reduced the problem to counting incidences only inside factors of a polynomial $Q$ of degree satisfying \eqref{d1}, does not take into account incidences in common factors of $Q$ and $\prod_{h\in \mathcal H} Q_h$. As a result,  the modified assumption \eqref{contr} forces one to have a rich irreducible degree $2$ factor of $Q$, which is not forbidden. This corresponds, within Theorem \ref{gkt2} to $\Omega(C\sqrt{n})$ lines from the family $L_\beta$ lying inside the intersection of $\Ge$ with some three-hyperplane $H\not\in \mathcal H^*$. In terms of Theorem \ref{mishh} this means collinearity of  $\Omega(C\sqrt{n})$ planes in $\Pi$ along some line not in $L^*$. This establishes the estimate \eqref{pupss}.\end{proof}

\renewcommand{\thetheorem}{\arabic{theorem}}
\addtocounter{theorem}{9}

Throughout the rest of the section, $\F$ is a field of odd characteristic $p$.

\subsection{On distinct values of bilinear forms}\label{vpr}
Established sum-product type inequalities over fields with positive characteristic have been weaker than over $\R$, where one can take advantage of the order structure and use geometric, rather than additive combinatorics. See, e.g.,  \cite{E0}, \cite{So},  \cite{KR}, \cite{BJ}, \cite{KS} for  some key methods and ``world records''.

The closely related geometric problem discussed in this section is one of  lower bounds on the cardinality of the set of values of a non-degenerate bilinear form $\omega$, evaluated on pairs of points from a set $S$ of $N$ non-collinear points in the plane. One may conjecture the bound $\Omega(N)$, possibly modulo factors, growing slower than any power of $N$. This may clearly hold in full generality in positive characteristic only if  $N=O(p)$.

The problem was claimed to have been solved over $\R$ up to the factor of $\log N$ in \cite{IRR}, $\omega$ being the cross or dot product. However, the proof was flawed. The error  came down to ignoring the presence of nontrivial weights or multiplicities, as they appear below. The best bound over $\R,\C$  that the erratum \cite{IRRE} sets is $\Omega(N^{9/13})$, for a skew-symmetric $\omega$. The bound $\Omega(N^{2/3})$ for any non-degenerate form $\omega$ follows just from applying the Szemer\'edi-Trotter theorem to bound the number of realisations of any particular nonzero value of $\omega$.

 \medskip
In this section we prove the following theorem.
\begin{theorem}\label{spr}
Let $\omega$ be a non-degenerate symmetric or skew-symmetric bilinear form and the set $S\subseteq \F^2$ of $N$ points not be supported on a single line.
Then
\begin{equation}|\omega(S) :=  \{\omega(s,s'):\,s,s'\in S\}| = \Omega\left[\min \left(N^{\frac{2}{3}},p\right)\right].\label{worst}\end{equation}
If $S$ has a subset $S'$ of $N'<p$ points, lying in distinct directions from the origin, then  $|\omega(S)|\gg N'.$ 
\end{theorem}

\begin{proof}  
From now on we assume that $S$ does not have more than $N^{\frac{2}{3}}$ points on a single line through the origin, for since $S$ also contains a point outside this line, the estimate \eqref{worst} follows. This assumption will be seen not to affect the second claim of the theorem.  Suppose also, without loss of generality, that $S$ does not contain the origin, nor does it have points on the two coordinate axes.

We may assume that $\F$ is algebraically closed, in which case one may take a symmetric form $\omega$ as given by the $2\times 2$ identity matrix and a skew-symmetric one by the canonical symplectic matrix. We consider the latter situation only. The former case is similar. One can also replace $S$ with its union with $S^\perp=\{(-q_2,q_1):\,(q_1,q_2)\in S\}$ and repeat the forthcoming argument.

Consider the equation
\begin{equation}\label{eng} \omega(s,s')=\omega(t,t')\neq 0, \qquad (s,s',t,t')\in  S\times S\times S\times S. \end{equation}

Assuming that $\omega$ represents wedge products, this equation can be viewed as counting the number of incidences between the set of points $P\subset \Pro^3$ with homogeneous coordinates $(s_1:s_2:t_1:t_2)$ and planes in a set $\Pi$ defined by covectors $(s_2':-s_1':-t_2':t_1')$. However, both points and planes are weighted. Namely, the weight 
$w(p)$ of a point $p=(s:t)$ is the number of points $(s,t)\in \F^4$, which are projectively equivalent that is lie on the same line through the origin. The same applies to planes. The total weight of both sets of points and planes is $W=N^2$. Like in the case of the Szemer\'edi-Trotter theorem, the weighted variant of  estimate of Theorem \ref{mish} gets worse with maximum possible weight.

The number of solutions of \eqref{eng}, plus counting also quadruples yielding zero values of $\omega$ is the number of weighted incidences
\begin{equation}\label{weightin}
I_w := \sum_{q\in P, \pi\in \Pi} w(q)w(\pi) \delta_{q\pi},
\end{equation}
where $\delta_{q\pi}$ is $1$ when $q\in \pi$ and zero otherwise.

 \medskip
Consider two cases:  {\em (i) $S$ only has points in $O(N^{2/3})$ distinct directions through the origin; (ii) there exists $S'\subset S$ with exactly one point in $\Omega(N^{2/3})$ distinct directions.}

\medskip
To deal with (i) we need the following weighted version of Theorem \ref{mish}. 
\begin{theorem}\label{wmish} 
 Let $P, \Pi$ be  weighted sets of points and planes  in  $\Pro^3$, both with total weight $W$. Suppose, maximum weights are bounded by $w_0>1$. Let $k$ be the maximum number of collinear points, counted without weights. Suppose, $\frac{W}{w_0}=O(p^2)$, where $p>2$ is the characteristic of $\F$.
Then the number $I_w$ of weighted incidences is bounded as follows:
\begin{equation}\label{pupsweight}
I_w=O\left( W\sqrt{w_0W}+ k w_0 W\right).\end{equation}
The same estimate holds for the quantity $I^*_w$, which discounts weighted incidences along a certain set $L^*$ of lines in $\Pro^3$, the quantity $k^*$, denoting the maximum number of points  incident to a line not in $L^*$ replacing $k$ in estimate \eqref{pupsweight}.
\end{theorem}

\begin{proof} It is a simple weight rearrangement argument, the same as, e.g., in \cite{IKRT} apropos of the Szemer\'edi-Trotter theorem.  Pick a subset $P'\subseteq P$, containing $n=O\left(\frac{W}{w_0}\right)$ richest points in terms of non-weighted incidences. Assign to each one of the points in $P'$ the weight $w_0$, delete the rest of the points in $P$, so $P'$ now replaces $P$. The number of weighted incidences will thereby not decrease. Now of all planes pick a subset  $\Pi'$ of the same number $n$ of the richest ones, in terms of their non-weighted incidences with $P'$. Assign once again the weight $w_0$ to each plane in  $\Pi'$. We now replace $P,\Pi$ with $P',\Pi'$ -- the sets of respectively $n$ points and planes, for which we apply Theorem \ref{mish}, counting each incidence $w_0^2$ times. Note that we may still have $k$ collinear points in $P'$ or planes in $\Pi'$. This yields \eqref{pupsweight}.  

For the last claim of Theorem \ref{wmish} we use Theorem \ref{mishh} instead of Theorem \ref{mish}.\end{proof}

Returning to the proof of Theorem \ref{spr}, suppose we are in case  (i). We will  apply the $I_w^*$ estimate of Theorem \ref{wmish}  to the weighted arrangement of planes and points in $\Pro^3$, representing \eqref{eng}. Let us show that the quantity $k^*$ can be bounded as $O(N^{\frac{2}{3}})$, after it becomes clear what the set $L^*$ of forbidden lines is. The quantity $k$ is the maximum number of collinear points in the set  $S\times S\in \F^4$, viewed projectively. Suppose, $k\geq N^{\frac{2}{3}}$. This means we have a two-plane through the origin in $\F^4$, which contains points of $S\times S$ in at least $N^{\frac{2}{3}}$ directions in this plane. If this two-plane projects on the first two coordinates in $\F^4$ one-to-one, then $S$ itself has points in  $N^{\frac{2}{3}}$ directions. But in case (i) this is not the case.

We now define the finite set $L^*$ of forbidden lines in $\Pro^3$ as two-planes in $\F^4$, which are Cartesian products of pairs of lines through the origin in $\F^2$, each supporting a point of $S$. Hence $k^*$ is the maximum number of points incident to any other line in $\Pro^3$. 
If the two-plane through the origin in $\F^4$ projects on each coordinate two-plane $\F^2$, containing $S$ as a line through the origin, it is a Cartesian product of two lines $l_1$ and $l_2$ through the origin in $\F^2$. Such a plane may contain a point $(s_1,s_2,t_1,t_2)\in \F^4$ or be incident to a three-hyperplane through the origin in $\F^4$, defined by the covector  $(s_2', -s_1',  -t_2', t_1')=0$ only if the lines $l_1,l_2$ contain points of $S$.

Applying the $I_w^*$-version of estimate \eqref{pupsweight}, we therefore obtain
\begin{equation}\label{weste}
I^*_w=O\left(N^{\frac{10}{3}} + N^{\frac{10}{3}}\right).
\end{equation}

It remains to show that point-plane incidences along the lines in $L^*$ correspond to zero values of the form $\omega$ in \eqref{eng}. By definition, a line in $L^*$ is represented by a pair $(l_1,l_2)$ lines through  the origin in $\F^2$. If the $\F^4$-point $(s,t)=(s_1,s_2,t_1,t_2)$ lies in the two-plane, which is the Cartesian product $l_1\times l_2$, this means $s\in l_1$, $t\in l_2$. If a three-hyperplane through the origin in $\F^4$, defined by the covector  $(s_2', -s_1',  -t_2', t_1')=0$ contains both lines $l_1,l_2$, this means $s'\in l_1$, $t'\in l_2$. Hence $\omega(s,s')=\omega(t,t')=0.$

So, if case (i) takes place, the bound \eqref{worst}  follows from \eqref{eng} and \eqref{weste} by the Cauchy-Schwarz inequality. Observe that Theorem \ref{wmish} applies under the constraint $N\leq cp^{\frac{3}{2}}$ for some absolute $c$. In particular, when $N= \lfloor cp^{\frac{3}{2}}\rfloor$,  it yields $I_w=O(p^5)$, hence one has $\Omega(N^{\frac{2}{3}})= \Omega(p)$ distinct values of the form $\omega$. For  $N\geq cp^{\frac{3}{2}}$ we do no more than retain this estimate.

Finally, if case (ii) takes place, we apply Theorem \ref{mish} to the set $S'$. For now planes and points bear no weights other than $1$, and the above argument about collinear planes and points applies. Namely one can set $k=N'$ and zero values of $\omega$ may no longer be excluded. Then equation \eqref{eng} with variables in $S'$ alone has $O({N'}^3)$ solutions, and the last claim of Theorem \ref{spr} follows by the Cauchy-Schwarz inequality.
\end{proof}

It is easy to adapt the proof of Theorem \ref{spr} to the special case when $S=A\times B$ for then one can set $w_0 = \min(|A|,|B|)$.  This results in the following corollary. There is also a more economical way of deriving the following statement from Theorem \ref{mish}. See \cite{RRS}, Corollary 4.
\begin{corollary}\label{hbk} Let $A,B\subseteq \F$, with $|A|\geq|B|$.
 Then
\begin{equation} |AB\pm AB|=\Omega\left[\min\left(|A|\sqrt{|B|},p\right)\right].\label{mebd}\end{equation}
\end{corollary}

\subsection{Tightness of Theorem \ref{mish}} \label{example}
We use the considerations of the previous section, looking at the number of distinct dot products of pairs of vectors in the set 
$$
S=\{(a,b):\,a,b\in[1,\ldots, n]:\; \mbox{gcd}(a,b)=1\}.
$$
The set can be thought of lying in $\R^2$ or $\F_p^2$, for $p\gg n^2$. Clearly,  $S$ has $N=\Theta(n^2)$ elements.

But now there are no weights in excess of $1$, in the sense of the discussion in the preceding section. So we can apply the argument from case (ii) within the proof of Theorem \ref{spr}  and get a $O(N^3)$ bound for the number of solutions $E$ of the equation, with the standard dot product,
\begin{equation}
s\cdot s' = t\cdot t', \qquad (s,s',t,t')\in S\times S\times S\times S.
\label{integers}\end{equation}
Note that zero dot products can only contribute $O(N^2)$.

On the other hand,  the same, up to constants, bound for $E$  from below follows by the Cauchy-Schwarz inequality.  Indeed, $x=s\cdot s' $ in equation \eqref{integers} assumes integer values in $[1\ldots4n^2]$. If $n(x)$ is the number of realisations of $x$, one has
$$E=\sum_x n^2(x) \geq \frac{1}{4n^2} \left(\sum_x n(x)\right)^2 \gg n^6\gg N^3.$$

\subsection{On distinct distances in $\F^3$}
\newcommand{\es}{\boldsymbol s}
\newcommand{\te}{\boldsymbol t}
Once again in this section $\F$ is an algebraically closed field of positive characteristic $p>2$.

The Erd\H os distance conjecture is open in $\R^3$, where it claims that a set $S$ of $N$ points determines $\Omega(N^{\frac{2}{3}})$ distinct distances\footnote{The conjecture is often formulated more cautiously, that there are $\Omega^*(N^{\frac{2}{3}})$ distinct distances, the symbol $\Omega^*$ swallowing terms, growing slower than any power of $N$.}.  The best known bound in $\R^3$ is $\Omega(N^{.5643})$, due to Solymosi and Vu \cite{SV}.

We prove the bound $\Omega(\sqrt{N})$ for the positive characteristic pinned version of the problem, i.e., for the number of distinct distances, attained from some point $\es\in S$, for $N=O(p^2)$, assuming that $S$ is not contained in a single semi-isotropic plane, as described below. 

Define the distance set
$$\Delta(S) = \{\|\es-\te\|^2:\,\es,\te \in S\},$$
with the notation $\es=(s_1,s_2,s_3)$, $\|\es\|^2 = s_1^2+s_2^2+s_3^2.$ Let us call a pair $(\es,\te)$ a {\em null-pair} if $\|\es-\te\|=0$.

In positive characteristic, the space $\F^3$ (even if $\F=\F_p$) always has a cone of {\em isotropic directions} from the origin,  that is $\{\om\in \F^3:\,\om\cdot \om=0\}$, with respect to the standard dot product.  See \cite{HI}, in particular Theorem 2.7 therein for explicit calculations of isotropic vectors and their orthogonal complements over $\F_p$.

The equation for the isotropic cone through the origin in $\F^3$ is clearly
\begin{equation}\label{cone}
x^2+y^2+z^2=0.
\end{equation}
It is a degree two ruled surface, whose ruling is not a regulus, see Section \ref{rgl}. 

If $\ee_1$ is an isotropic vector through the origin, its orthogonal complement $ \ee_1^\perp$ is a plane, containing $\ee_1$. Let $\ee_2$ be another basis vector in this plane, orthogonal to $\ee_2$. Then $\ee_2$ is not isotropic, for otherwise the whole plane $\ee_1^\perp$ would be isotropic. This is impossible, for equation \eqref{cone} is irreducible. We call the plane $\ee_1^\perp$ or its translate {\em semi-isotropic.}

The fact that $\ee_2$ is not isotropic implies that there are no {\em nontrivial null triangles} that is triangles with three zero  length sides, unless the three vertices lie on an isotropic line. With this terminology, there exist only {\em trivial} null triangles in $\F^3$.

In a semi-isotropic plane one can have $N=kl$ points, with $1\leq k\leq l$, with just $O(k)$ distinct pairwise distances: place $l$ points on each of $k$ parallel lines in the direction of $\ee_1$, whose $\ee_2$-intersects are in arithmetic progression.

To deal with zero distances we use the following lemma.
\begin{lemma} \label{easy} 
Let $T$ be a set of $K$ points on the level set 
$$Z_R=\{(x,y,z):\,x^2+y^2+z^2=R\}.$$ For $K\gg1$ sufficiently large, either $\Omega(K)$ points in $T$ are collinear, or a possible proportion of $(\te,\te')\in T\times T$ are not null pairs. 
\end{lemma}
\begin{proof} First note, as an observation, that even if $R\neq 0$, when $Z_R$ is a doubly-ruled quadric it may well be ruled by isotropic lines.  Indeed, representing lines in $\F^3$ by Pl\"ucker vectors $(\om:\ve)$ in the Klein quadric $\Ka$, defined by the relation \eqref{Klein}, i.e., $\om\cdot\ve =0$, isotropic vectors are cut out by the quadric 
$\om\cdot\om=0$, while a regulus is a conic curve cut out from $\Ka$ by a two-plane. If the intersection of the three varieties in question is non-degenerate, it is at most four points, that is there are at most four isotropic lines per regulus.

However, take the two-plane as $\ve =\lambda \om$, for some $\lambda\neq 0$. (The case $\lambda=0$ corresponds to the isotropic cone through the origin.) Write $\ve=ad(\qu)\,\om$, for some point $\qu\in \F^3$ lying on the line in question, where $ad(\qu)$ is a skew-symmetric matrix, see Footnote \ref{ads}. This yields the eigenvalue equation $\det(ad(\qu) - \lambda I) =0,$ which means that $\qu$ satisfies
$$
\lambda^2+ \|\qu^2\|=0,
$$
that is  $\qu\in Z_{-\lambda^2},$  the point set of a regulus of isotropic lines.

\medskip
Turning to the actual proof of the lemma, consider a simple undirected graph $G$ with the vertex set $T$, where there is an edge connecting distinct vertices $\te$ and $\te'$ if $(\te,\te')$ is a null pair. Suppose $G$ is close to the complete graph, that is $G$ has at least  $.99 K(K-1)/2$ edges. 

Note that $K'\leq K$ points of $T$, lying on an isotropic line,  yield a clique of size $K'$ in $K$. Suppose there is no clique of size, say $K'\geq.01K,$ or we are done. 

Then in each clique of size $K'$ one can delete at most $(K'+1)^2/4$ edges,  turning it into a bipartite graph, whereupon there are no triangles left within that clique. After that one is left with no triangles in $G$, corresponding to trivial null triangles in $Z_R$.

However if   $K'\leq .01K$, the number of remaining edges is still greater than $K^2/4$, clearly the former cliques had no edges in common. By Turan's theorem there is a triangle in what is left of $G$, and it corresponds to a nontrivial null triangle in  $Z_R$.

This contradiction finishes the proof of Lemma \ref{easy}
 \end{proof}

We are now ready to prove the last theorem in this paper.
\begin{theorem}\label{erd} A set $S$ of $N$ points in $\F^3$, such that all points in $S$ do not lie in a single semi-isotropic plane, determines $\Omega[\min(\sqrt{N},p)]$ distinct pinned distances, i.e., distances from some fixed $\es\in S$ to other points of $S$.
\end{theorem}

\begin{proof}
First off, let us restrict $S$, if necessary, to a subset of at most $cp^2$ points, where $c$ is some small absolute constant, later to enable us to use Theorem \ref{mish}. We keep using the notation $S$ and $N$. Furtermore, we assume that $S$ has at most $\sqrt{N}$ collinear points or there is nothing to prove: even if $\sqrt{N}$ collinear points lie on an isotropic line, $S$ has another point $\es$  outside this line, such that the plane containing $\es$ and the line is not semi-isotropic. It is easy to see that then there are $\Omega(\sqrt{N})$ distinct distances from $\es$ to the points on the  line.

Let $E$ be the number of solutions of the equation
\begin{equation}\label{energy} 
\|\es-\te\|^2 = \|\es-\te'\|^2\neq 0,\qquad (\es,\te,\te')\in S\times S\times S.\end{equation}
Let us show that either $S$ contains a line with $\Omega(\sqrt{N})$ points or
\begin{equation}\label{clm}
E=O(N^{\frac{5}{2}}).
\end{equation}

We claim, by the pigeonhole principle and  Lemma \ref{easy}, that assuming $E\gg N^{5/2}$ implies that either there is a line with $\Omega(\sqrt{N})$ points, or $E=O(E^*)$, where $E^*$ is the number of solutions of the equation
\begin{equation}\label{energystar} 
\|\es-\te\|^2 = \|\es-\te'\|^2\neq 0,\qquad (\es,\te,\te')\in S\times S\times S:\;\|\te-\te'\|\neq 0.
\end{equation}

Indeed, the quantity $E$ counts the number of equidistant pairs of points from each $\es\in S$ and sums over $\es$. Therefore, a positive proportion of $E$ is contributed by points $\es$ and level sets $Z_R(\es)=\{\te\in \F^3:\, \|\es-\te\|=R\}$, such that $Z_R(\es)$ supports $\Omega(\sqrt{N})$ points of $S$. 
By Lemma \ref{easy} either there is a line with $\Omega(\sqrt{N})$ points, or  a positive proportion of pairs of distinct  $\te,\te'\in Z_R(\es)$ is non-null.

This establishes the claim in question.

\medskip
Now observe that to evaluate the quantity $E^*$, for each pair $(\te,\te')$ we have a plane through the midpoint of the segment $[\te\,\te']$, normal to the vector $\te-\te'$ and need to count points $\es$ incident to this plane. The plane in question does not contain $\te$ or $\te'$.

We arrive at an incidence problem $(S,\Pi)$ between $N$ points and a family of planes, but the planes have weights in the range $[1,\ldots,N]$, for the same plane can bisect up to $N/2$ segments $[\te \,\te']$, {\em provided that} $(\te,\te')$ is not a null pair. That is given the plane, there is at most one $\te'$ for each $\te$, so that the plane may bisect  $[\te \,\te']$.  

Thus number $m$ of distinct planes is $\Omega(N)$ and at most $N^2$, the maximum weight per plane is $N$, the total weight of the planes $W=N^2$.

It is immediate to adapt the formula (\ref{pups}) to the case of planes with weights. Note that the number of distinct planes is not less than the number of points, so in the formula \eqref{pups}, the notation $m$ will now pertain to planes, $n$ to points, and $k$ to the maximum number of collinear points. Since the estimate (\ref{pups}) is linear in $m$, the case of weighted planes and non-weighted points arises by replacing $m$ with $N^2$, $n$ with $N$, and $k$ with $\sqrt{N}$, for otherwise, once again, there is nothing to prove.

Theorem \ref{mish} now applies for $N=O(p^2)$ and yields the estimate \eqref{clm}. Theorem \ref{erd} follows from \eqref{energy} by the Cauchy-Schwarz inequality.  In particular, when $N=cp^2$ for some absolute $c$, we get $\Omega(p)$ distinct pinned distances. If  $N\geq cp^2$ we simply retain this estimate.

\end{proof}

\end{document}